\documentclass[12pt]{article}
\usepackage{mathtools}
\usepackage[english]{babel}
\usepackage[utf8]{inputenc}
\usepackage[T1]{fontenc}

\usepackage{amsmath, amsthm, amssymb}
\usepackage[colorlinks=true,
            linkcolor=blue,
            urlcolor=blue,
            citecolor=blue]{hyperref}
\usepackage[shortlabels]{enumitem}

\providecommand{\keywords}[1]{\textbf{Keywords.} #1}
\providecommand{\MSC}[1]{\textbf{AMS Subject Classifications.} #1}

\usepackage {amsmath,amssymb,verbatim,geometry}
\usepackage[all]{xy}
\usepackage{array}

\newif\ifpdf
\ifpdf
\usepackage[pdftex]{hyperref}
\else
\fi

\numberwithin{equation}{section}       

\newtheorem{prop} {Proposition} [section]
\newtheorem{thm}[prop] {Theorem}

\newtheorem{defi}[prop] {Definition}
\newtheorem{lem}[prop] {Lemma}

\newtheorem{prop-def}[prop]{Proposition-Definition}

\newtheorem*{thmA}{Theorem A}

\newtheorem*{thmB}{Theorem B}

\newtheorem{rem}[prop]{Remark}

\newcommand{\C}{{\mathbb{C}}}
\newcommand{\B}{{\mathbb{B}}}

\newcommand{\R}{{\mathbb{R}}}

\newcommand{\psh}{{\mathrm{PSH}}}

\title{ Viscosity solutions to complex first eigenvalue equations}
\author{
Soufian Abja\thanks{Institute of Mathematics, Jagiellonian University, Lojasiewicza 6,
30-348 Krakow, Poland (\href{mailto:Soufian.Abja@im.uj.edu.pl}{Soufian.Abja@im.uj.edu.pl}).}

}

\date{}
\begin{document}
\maketitle

\begin{abstract}
We study the viscosity solutions to the first eigenvalue equation. We consider  $\Omega$ a bounded B-regular domain in $\C^n$ and we prove that the Dirichlet problem $\Lambda_{1}(D_{\C}^2 u)=f$ in $\Omega$ and $u=\varphi$ on $\partial\Omega$ admits a unique viscosity solution. We also deal with viscosity theory for operators which are comparable to the first eigenvalue operator.
\end{abstract}


\vspace{0.2cm}
\MSC{35J70, 35D40}
\section{Introduction}
Viscosity methods provide a powerful tool for the study of non-linear degenerate partial differential equations. They were introduced in the early eighties by Crandall-Lions \cite{cl} as a generalization of the classical solutions to partial differential equations (see the survey \cite{crl}, and references therein).

Recently viscosity methods were applied for the complex Monge-Amp\`ere equations by Eyssidieux-Guedj-Zeriahi \cite{egz11} and Wang \cite{wang}. They were also used for the complex Hessian equations by Lu \cite{lu} and for a more general type of complex Hessian equations by Harvey-Lawson \cite{hl09} and Dinew-Do-T{\^o} \cite{ddt}. 

The aim of this paper is to use the viscosity theory to study an equation that has not been given much attention in the literature especially in the complex setting. The equation is the first eigenvalue of the complex Hessian, its Dirichlet problem is given by
\begin{equation}\label{main}
 \begin{cases}
  \Lambda_1(D_{\C}^2 u)= f(z),& \;\text{in } \Omega,\\
      
     u=\varphi, & \;\text{on}\; \partial\Omega.
    \end{cases}
    \end{equation}
    Here $\Omega$ is a bounded domain in $\C^n$ and $\Lambda_1$ stands for the smallest eigenvalue of
$D_{\C}^2u =\left(\frac{\partial^2 u}{\partial z_j\partial\bar{z}_k} \right)_{1\leq j,k\leq n}$, which is the complex Hessian matrix of $u$, $f$ is a positive continuous function in the closure of $\Omega$ and  $\varphi$ is a continuous function on the boundary of $\Omega$.

The study of the first eigenvalue equation in the real setting was initiated  by Oberman \cite{adam} and Oberman-Silvestre \cite{os}. The 
 partial sums of eigenvalues were studied in \cite{sha} and \cite{wu}, which had been motivated by some early geometric problems. They have also been considered  in \cite{caffarelli}, \cite{amendola}, \cite{hl13}, \cite{vitolo}, \cite{capuzzo}, \cite{pj} and \cite{ferrari}. 

In the complex setting the only related work is the paper of Rashkovskii \cite{rash}. He proved that for a positive current $T$ of bi-degree (1,1) on a domain $\Omega\subset\C^n$ written as follows: \[ T=i\sum_{i,j=1}^nT_{i\bar{j}} dz_i\wedge d\bar{z}_j,\] the eigenvalue operators $\Lambda_k: T\mapsto \Lambda_k(T)$ of the the Hermitian form $T_{j\bar{k}}$ are not continuous in the weak topology even on closed positive 
currents $T$ in spite of continuity of $\sum_{k=1}^{n}\Lambda_k.$ The definition of $\Lambda_k$ in \cite{rash} is through the Goffman-Serrin  construction \cite{gose}. It differs from the viscosity definition that we will use in this paper.

There are several reasons for studying the least eigenvalue on the complex analysis. Among them let us mention the following ones. Control on the first eigenvalue leads to the construction of strictly plurisubharmonic functions. These are in particular related to the existence of subelliptic estimates for the $\bar{\partial}$-Neumann problem see \cite{subelliptic}. Least eigenvalues also appear in exact bounds for solutions to various $\bar{\partial}$-equations with weights (see \cite{de}). Yet another motivation comes from the recent papers \cite{ADO20} and \cite {AO21} where authors studied general complex equations which satisfied some assumptions, but the key assumption was the comparability with complex  Monge-Amp\`ere operator i.e.
\begin{equation}
G(P)\geq C \det(P)^{\frac{1}{n}}, \forall P\in \mathcal{C}_n,
\end{equation}
where $\mathcal{C}_n=\{ A\in\mathbb{H}\;|\; A>0\}$ and $C>0$ is constant.
In the above inequality if we replace $\det^{\frac{1}{n}}$ by $\Lambda_1$ i.e.
\begin{equation}
G(P)\geq C \Lambda_1(P), \forall P\in \mathcal{C}_n,
\end{equation}
we can enlarge the class of the equations studied there.

We are now in the position to announce our main results. Our first main result is the  following comparison principle.
\begin{thmA}\label{maximum}Let $u$ be a bounded viscosity subsolution and $v$ be a bounded viscosity supersolution of (\ref{main}). If $u\leq v$ on $\partial\Omega$ then $u\leq v$ on $\Omega$.
\end{thmA}

When we have the comparison principle, the existence of the solution to the Dirichlet problem (\ref{main}) can be established by the classical Perron method modulo the existence of subsolution and supersolution.
To construct a subsolution to the Dirichlet problem (\ref{main}), we restrict our attention to the class of domains that admit strong plurisubharmonic barriers. Such domains were investigated by Sibony \cite{sib} and are called  $B$-regular domains. 
\begin{thmB}\label{thmb}Let $\Omega$ be a $B$-regular domain. The Dirichlet problem (\ref{main}) admits a unique viscosity solution $u\in C(\overline{\Omega})$.
\end{thmB}

The paper is organized as follows: In Section 2, we collect some preliminaries. In Section 3, we define and develop the viscosity theory for the first eigenvalue operator and we prove our main results. In Section 4, we deal with viscosity theory for operators that are comparable to the first eigenvalue.

\section{Preliminaries}
In this section, we introduce the definitions and notations that we need through our paper.

Let $\mathbb{H}$ denote the set of all Hermitian  $n\times n$ matrices. We  denote also by $\mathbb{S}^{2n}$  the set of all symmetric $2n\times 2n$ matrices. For any Hermitan matrix $A$ we denote by $\Lambda_{k}(A)$, $k=1,...,n$ its eigenvalues and order them so that $\Lambda_1\leq ...\leq \Lambda_n$. We have the following formulas:
\begin{eqnarray}\label{egen}
 \Lambda_k(A)&=&\max_{\substack{M\subset \C^n\\ \dim M=k}}\min_{\substack {X\in M \\||X||=1}}\langle X,AX\rangle \nonumber\\
  &=&\min_{\substack{M\subset \C^n\\ \dim M=n-k+1}}\max_{\substack {X\in M \\||X||=1}}\langle X,AX\rangle,
\end{eqnarray}
where $\langle,\rangle$ is the usual inner product on $\C^n$ and $||.||$ is its associated norm. The above equality is well known in the literature see \cite[Corollary III.1.5]{bhatia}) for example.

The following lemma summarizes the properties of the first eigenvalue.
\begin{lem}\label{eigenvalues} The eigenvalues satisfy the following properties
\begin{itemize}
\item $\Lambda_k(\alpha A)=\alpha\Lambda_k(A)$ for all $\alpha>0$ and $A\in \mathbb{H}$.
\item $\Lambda_k(A+B)\geq \Lambda_k(A)+\Lambda_1(B)$ for all $A,B\in \mathbb{H}$.
\item $\Lambda_k(A+B)\leq\Lambda_k(A)+\Lambda_n(B)$ for all $A,B\in \mathbb{H}$.
\item $\Lambda_k$ are elliptic operators, which means
$$\Lambda_k(A+P)\geq \Lambda_k(A),$$ 
for all $A\in \mathbb{H}$ and $P\geq 0$.
\end{itemize}
\end{lem}
\begin{proof}The first statement is a direct consequence of (\ref{egen}). The second and third are known in the literature as Weyl's inequalities (see \cite[Corollary III.2.2]{bhatia}). The last one is a consequence of the second one.
\end{proof}
\begin{rem}
\begin{itemize}As a corollary of the above lemma the following two statements hold:
\item $A \mapsto \Lambda_1(A)$ is concave i.e.
 $$\Lambda_1(\alpha A+(1-\alpha)B)\geq \alpha\Lambda_1(A)+(1-\alpha)\Lambda_1(B),$$ for all $A,B\in \mathbb{H}$ and  $\alpha\in(0,1).$
 \item $A \mapsto \Lambda_n(A)$ is convex i.e.
 $$\Lambda_n(\alpha A+(1-\alpha)B)\leq \alpha\Lambda_n(A)+(1-\alpha)\Lambda_n(B),$$ for all $A,B\in \mathbb{H}$ and  $\alpha\in(0,1).$
\end{itemize}
\end{rem}
We will denote throughout this paper by $\mathbb{B}(z,r)$ the ball of center $z$ and radius $r$ and by
$$u^{*}(z)=\inf_{r>0}\sup\{u(y); y\in\B(z,r)\},$$
$$u_{*}(z)=\sup_{r>0}\inf\{u(y); y\in\B(z,r)\},$$
the lower and upper semi-continuous envelopes of $u$.

\begin{defi}Let $\Omega$ be a domain in $\C^n$. A function  $u:\Omega\rightarrow \R$ is called plurisubharmonic if it is upper semi-continuous, $L^{1}_{loc}$ and satisfies $$dd^c u\geq 0$$
in weak sense of positive currents, where $d:=\partial+\bar{\partial}$ and $d^c:=\frac{i}{2\pi}(\partial-\bar{\partial})$.

 The space of plurisubharmonic functions will be denoted by $\psh(\Omega)$.
\end{defi}
A bounded domain $\Omega$ in $\C^n$ is said to be a $B$-regular if one of the following equivalent statements is satisfied.
\begin{thm}\cite{sib}\label{sib}For a bounded domain in $\C^n$ the following are equivalent:
\begin{itemize}
\item For every $z_0\in\partial\Omega$, there exists $u\in \psh(\Omega)$ such that 
$$\lim_{z \rightarrow z_0} u(z)=0 \; \text{and} \;u^{*}|_{\overline{\Omega}\setminus \{ z_0\}}<0.$$

\item In $\Omega$ there exists a smooth exhaustive function such that 
$$\sum_{j,k=1}^n \frac{\partial^2 \psi}{\partial z_j\partial\bar{z}_k}\xi_j\bar{\xi_k} \geq |\xi|^2, \xi\in \C^n.$$
\item  For every $f\in C(\partial \Omega)$ there exists $u\in \psh(\Omega)\cap C(\overline{\Omega})$ such that $u|_{\partial \Omega}=f$.
\end{itemize}
For the proof of this theorem see \cite{bl96}.
\end{thm}
\section{Viscosity solution}
Our goal here is to study the first eigenvalue equation in the viscosity sense. We consider the equation
\begin{equation}\label{first} 
      \Lambda_1(D_{\C}^{2}u)(z)=f(z),\; \text{in}\; \;\Omega,
\end{equation}
where $\Omega\subset\C^n$ is an $B$-regular domain, and $ f: \overline{\Omega}\rightarrow \R$ is a positive continuous function. The unknown function $u$ is a plurisubharmonic function $ u:\Omega\rightarrow \R$.

The equation (\ref{first}) can also be written as
\begin{equation}
\Lambda_1(D_{\C}^{2}u)(z)-f(z)=0,\; \text{in}\; \Omega.
\end{equation}
This equation is degenerate elliptic when we restrict considerations  to space of plurisubharmonic functions.  Identifying $\C^n$  with $\R^{2n}$, we set 
 $$F(z,Q):=
 \begin{cases}
  \Lambda_1(Q^{1,1})-f(z)  ,& \;\text{if } Q^{1,1}\geq 0,\\
      -\infty, & \;\text{otherwise}, 
    \end{cases}$$
where $(z,Q)\in \Omega\times \mathbb{S}^{2n}$ and 
$Q^{1,1}\in \mathbb{H}^n$ is the Hermitian (1,1)-part of $Q\in \mathbb{S}^{2n}$ considered as a real quadratic form on $\C^n$. The function $F$ is lower semi-continuous on $\Omega\times \mathbb{S}^{2n}$ and continuous on its domains $\{F>-\infty\}=\Omega\times \{Q\in\mathbb{S}^{2n}: Q^{1,1}\geq 0\}$.

 We introduce now the definitions of viscosity subsolutions and viscosity supersolutions to the first eigenvalue equation.
\begin{defi}Let $u: \Omega\rightarrow \R\cup\{-\infty\}$ and let $\varphi$ be a $C^2$ in a neighborhood of $z_0\in\Omega$. We say that $\varphi$ touches $u$ from above at $z_{0}$ if $\varphi(z_0)=u(z_0)$ and 
$\varphi(z)\geq u(z)$ for every $z$ in a neighborhood of $z_0$.

 We say  that $\varphi$ touches $u$ from below at $z_{0}$ if $\varphi(z_0)=u(z_0)$ and 
$\varphi(z)\leq u(z)$ for every $z$ in neighborhood of $z_0.$
\end{defi}
\begin{defi}\label{1} An upper semi-continuous function  $u: \Omega \rightarrow \R\cup \{-\infty\}$ is a viscosity sub-solution to (\ref{first})
 if $u\neq -\infty$ and for any $z_0\in\Omega$ and any $C^2$ function $\varphi$ that touches $u$ from above at $z_0$ we have
 $$\Lambda_1(D_{\C}^2\varphi)\geq f(z)\;\text{at}\; z_0.$$
 
 A lower semi-continuous function  $u: \Omega \rightarrow \R\cup \{-\infty\}$ is a viscosity super-solution to (\ref{first})
 if $u\neq -\infty$ and for  any $z_0$ and any $C^2$ function $\psi$ that touches $u$ from below at $z_0$ we have
 $$[\Lambda_1(D_{\C}^2\psi)]^+\leq f(z)\; \text{at}\; z_0.$$
 Here $[\Lambda_1(D_{\C}^2\psi)]^+$  is defined to be $\Lambda_1(D_{\C}^2\psi)$ if it is positive and $0$ otherwise.
 
 A function $u: \Omega \rightarrow \R$ is a viscosity solution to (\ref{first}) if it is both a viscosity subsolution and a viscosity
supersolution to (\ref{first}). A viscosity solution is automatically a continuous function in $\Omega$.
\end{defi}
The following proposition follows easily from the definition of a subsoution.
\begin{prop}\label{glue}
Assume that $G\subsetneq \Omega$ is an open set. Suppose that $u$ is a viscosity
subsolution of (\ref{first}) in $G$ and $v$ is a viscosity subsolution of (\ref{first}) in $\Omega$ such that
$$\limsup_{z\rightarrow z_0} u(z)\leq v(z_0),$$
for every $z_0 \in \partial G \cap\Omega$. Then, the function
$$h(z)=
 \begin{cases}
      \max(u(z),v(z)),& \;\text{in}\; G,\\
      v(z) ,& \;\text{in}\;\Omega\setminus G
    \end{cases}$$
is a viscosity subsolution of (\ref{first}) in $\Omega$.
\end{prop}
Since the fundamental Jensen-Ishii  maximum principle will be stated in terms
of semi-jets, it is convenient to use these notions which we now recall following \cite{crl}, in order to characterize as well the notions of sub/super solutions.
\begin{defi} 1) Let $u: \Omega \rightarrow \R$ be an upper semi-continuous and $z_0\in\Omega$. The super-differential jet of $u$ at $z_0$ is given by
\begin{multline*}
J^{2,+}u(z_0)=\{(p,Q)\in \R^{2n}\times \mathbb{S}^{2n} \;:\;
u(z_0+\xi)\leq u(z_0)+<p,\xi>\\
+\frac{1}{2}<Q\xi,\xi>+o(|\xi|^2)\; \forall\xi\in \R^{2n} ,\; |\xi|<<1\}.
\end{multline*}
 2) Let $u: \Omega \rightarrow \R$ be a lower semi-continuous and $z_0\in\Omega$. The sub-differential jet of $u$ at $z_0$ is given by
 \begin{multline*}
$$J^{2,-}u(z_0)=\{(p,Q)\in \R^{2n}\times \mathbb{S}^{2n} \;/\;
u(z_0+\xi)\geq u(z_0)+<p,\xi>\\
+\frac{1}{2}<Q\xi,\xi>+o(|\xi|^2)\; \forall\xi\in \R^{2n} ,\; |\xi|<<1\}.
\end{multline*}
3) If $u$ is continuous we define the differential 2-jet as
$$J^{2}u(z_0)=J^{2,+}u(z_0)\cap J^{2,-}u(z_0).$$
\end{defi}
\begin{rem}If $u$ is twice differentiable at $z_0$ we have 
$$J^{2,+}u(z_0)=\{(Du(z_0),Q): Q\geq D^2u(z_0)\}$$
and $$J^{2,-}u(z_0)=\{(Du(z_0),Q): Q\leq D^2u(z_0)\}$$ 
so that $$J^{2}u(z_0)=\{(Du(z_0),D^2u(z_0))\}.$$ 
\end{rem}
By the continuity of $F$ in its domain we can extend the previous definition to the notion of approximate super-differential and sub- differential jets as follows:
\begin{multline*}
\bar{J}^{2,+}u(z_0)=\{(p,Q)\in \R^{2n}\times \mathbb{S}^{2n}: \exists (z_n,p_n,Q_n)\in \Omega\times \R^{2n}\times \mathbb{S}^{2n},(p_n,Q_n)\in J^{2,+}u(z_n)\\
\text{and}\; (z_n,p_n,Q_n)\rightarrow(z_0,p,Q)\}.
\end{multline*}
\begin{multline*}
\bar{J}^{2,-}u(z_0)=\{(p,Q)\in \R^{2n}\times \mathbb{S}^{2n}: \exists (z_n,p_n,Q_n)\in \Omega\times \R^{2n}\times \mathbb{S}^{2n},(p_n,Q_n)\in J^{2,-}u(z_n)\\
\text{and}\; (z_n,p_n,Q_n)\rightarrow(z_0,p,Q)\}.
\end{multline*}
From the above definition of a super-jet, we can equivalently define the subsolutions and supersolutions as in the following proposition.
\begin{prop}\label{2}
An upper semi-continuous function $u : \Omega \mapsto \R$ is a viscosity subsolution of (\ref{first}) if and only if 
for all $z_0\in\Omega,\; (p,Q)\in\bar{J}^{2,+}(z_0),$
$$\Lambda_1(Q)\geq f(z_0).$$
																			
A lower semi-continuous function $u : \Omega \mapsto \R$ is a viscosity supersolution of (\ref{first}) if and only if for all $z_0\in\Omega$, $(p,Q)\in\bar{J}^{2,-}(z_0)$, 
$$[\Lambda_1(Q)]^+\leq f(z_0).$$
\end{prop}
The proof of the above proposition goes as in \cite{crl} (see also \cite[proposition 6.11]{gz}). For more information on the viscosity theory in the complex setting, especially for complex Monge-Amp\`ere equations, see \cite{gz,zeriahi}.

The notion of viscosity subsolutions is stable under taking maximum. It
is also stable along monotone sequences as the following lemma shows.

\begin{lem}\label{convergence}Let $u_j$ be a monotone sequence of viscosity subsolutions of (\ref{first}). If $u_j$ is uniformly bounded from above and $u:=(\lim u_j)^{*}\neq-\infty$ then $u$ is also a subsolution of (\ref{first}).
\begin{proof}
This proof is classical and can be found in \cite{crl}. We adapt here the proof of Lu \cite{lu} for our operator. Take $z_0 \in \Omega$ and a $C^2$ function $\varphi$ in $\B(z_0,r)\subseteq\Omega$ which touches $u$ from above at $z_0$. We can choose a sequence $z_j$ in $B:= \bar{\B}(z_0,\frac{r}{2})$ converging
to $z_0$ and a subsequence of $u_j$ (denoted also by $u_j$) such that $u_j(z_j) \rightarrow u(z_0)$. Fix $\varepsilon>0$. For each $j$, let
$y_j$ be the maximum point of $u_j -\varphi -\varepsilon |y_j-z_0|^2$ in $B$. So 
\begin{equation}\label{convergence}
u_j(z_j)- \varphi(z_j)-\varepsilon|z_j-z_0|^2\leq u_j(y_j)-\varphi(y_j)-\varepsilon|y_j-z_0|^2.
\end{equation}
As $y_j\in B$, then there exists a subsequence of $(y_j)$ denoted also by $(y_j)$ such that $y_j \rightarrow y\in B$. Letting $j\rightarrow +\infty$ in  and noting that $\limsup u_j (y_j) \leq  u(y)$, we have
$$0 \leq u(y) - \varphi(y) - \varepsilon|y - z_0|^2.$$
Since $\varphi$ touches $u$ from above at $z_0$, we get that $y = z_0$. So  $y_j \rightarrow z_0$ as $j \rightarrow +\infty$. Then by (\ref{convergence}) again we
have $$u_j (y_j )\rightarrow u(z_0).$$ 
For $j$ sufficiently large, the function $\varphi(z) + \varepsilon|z-z_0|^2 + u_j (y_j)-\varphi(y_j)-\varepsilon|y_j- z_0|^2$
touches $u_j$ from above at $y_j$ . Since $u_j$ is a viscosity subsolution of (\ref{first}), by definition we have
$$\Lambda_1(D_{\C}^2(\varphi + \varepsilon|z-z_0|^2)(y_j ) \geq f(y_j).$$
Let $j \rightarrow +\infty$ to get
$$\Lambda_1(D_{\C}^2(\varphi + \varepsilon|z-z_0|^2)(z_0 ) \geq f(z_0).$$
Finally let $\varepsilon \rightarrow 0$ and use the continuity of $F$ to conclude that $\Lambda_1(D_{\C}^2(\varphi((z_0))\geq f(z_0)$.
Hence $u$ is also a viscosity subsolution of (\ref{first}).
\end{proof}
\end{lem}
\subsection{Comparison  principle}
A key ingredient for our approach is the following viscosity comparison principle, which implies the uniqueness of viscosity solutions to the first eigenvalue equation. Our approach here is based on \cite{crl} and also its analogue in the complex setting for the complex Monge-Amp\`ere equation \cite{egz}.

We assume that there exists $z_0\in \Omega$ such that 
\begin{equation}\label{maximum}
(u-v)(z_0)=\max_{\overline{\Omega}}(u-v)>0.
\end{equation}
First we consider the case where $u-\varepsilon|z-z_0|^2$ is a subsolution of (\ref{first}) for some $\varepsilon>0$. Set $u_{\varepsilon}(z):=u(z)-\frac{\varepsilon}{2}|z-z_0|^2$. 
We have $u_{\varepsilon}-v$ has a strict maximum at $z_0$. From the subadditivity of $\Lambda_1$ and our assumption that  $u-\varepsilon|z-z_0|^2$ is subsolution, we have
\begin{eqnarray}\label{strict}
\Lambda_1(D_{\C}^2u_{\varepsilon})(z)&=& \Lambda_1\left(D_{\C}^2\left(u(z)-\varepsilon|z-z_0|^2+\frac{\varepsilon}{2}|z-z_0|^2\right)\right)\nonumber\\
&\geq &  \Lambda_1\left(D_{\C}^2\left(u(z)-\varepsilon|z-z_0|^2\right)\right)+\Lambda_1\left(\frac{\varepsilon I}{2}\right)\nonumber\\
&\geq & f(z)+\frac{\varepsilon}{2}.
\end{eqnarray}
For each $\alpha>0$, we denote by 
$$\phi_{\alpha}(z,w)=u_{\varepsilon}(z)-v(w)-(1/2\alpha)|z-w|^2,$$
for all $(z,w)\in\overline{\Omega}\times\overline{\Omega}$. Since $\overline{\Omega}\times\overline{\Omega}$ is compact and $\phi_{\alpha}$ is upper semi-continuous, there exists $( z_{\alpha}, w_{\alpha})\in \overline{\Omega}\times\overline{\Omega}$
such that
$$\phi_{\alpha}( z_{\alpha},w_{\alpha})=\max_{(z,w)\in\overline{\Omega}\times\overline{\Omega}}\phi_{\alpha}(z,w).$$
From the classical work of \cite{crl}, we have the following lemma.
\begin{lem}\label{aa}We have 
\begin{equation}\label{max1}
|z_{\alpha}-w_{\alpha}|^2=o(\alpha).
\end{equation}
 Every limit point $(\hat{z},\hat{w})$ of $(z_{\alpha},w_{\alpha})$ as $\alpha \rightarrow 0$ satisfies $\hat{z}=\hat{w}$, $(\hat{z},\hat{z})\in\overline{\Omega}\times\overline{\Omega}$ and
 \begin{equation} \label{max2} 
\lim_{\alpha \rightarrow 0}\phi_{\alpha}( z_{\alpha}, w_{\alpha})=\lim_{\alpha \rightarrow 0}(u_{\varepsilon}(z_{\alpha})-v(w_{\alpha}))=u_{\varepsilon}(\hat{z})-v(\hat{z}).
\end{equation}
\end{lem}
The sequences $z_{\alpha}$ and $w_{\alpha} $ have subsequences which converge in $\overline{\Omega}$. From (\ref{max1}) we conclude that their limit are identical.
We have 
 $$\phi_{\alpha}(z,w)\leq \phi_{\alpha}(z_{\alpha},w_{\alpha}),\; \forall (z,w)\in \overline{\Omega}\times \overline{\Omega}.$$
 For $z=w=z_0$, we have
 $$u_{\varepsilon}(z_0)-v(z_0)\leq \phi_{\alpha}(z_{\alpha},w_{\alpha}).$$
 We use the upper semi-continuity of $u_{\varepsilon}$ lower semi-continuity of $v$ and (\ref{max1}) to conclude
$$ u_{\varepsilon}(z_0)-v(z_0)\leq \lim_{\alpha\rightarrow 0}\phi_{\alpha}(z_{\alpha},w_{\alpha})\leq u_{\varepsilon}(\hat{z})-v(\hat{z})\leq u_{\varepsilon}(z_0)-v(z_0).$$
  The above discussion with strict maximum of $u_{\varepsilon}-v$ at $z_{0}$ guarantee  that $\hat{z}=z_0\in \Omega$, because $u_{\varepsilon}(z_0)-v(z_0)=u(z_0)-v(z_0)>0$ and $u_{\varepsilon}\leq v$ on $\partial\Omega$.
  
   Therefore $(z_{\alpha},w_{\alpha})\rightarrow(z_0,z_0)$ as $\alpha\rightarrow 0$. Take any sequence$  (z_j,w_j):=(z_{\alpha_j},w_{\alpha_j})\rightarrow (z_0,z_0)$ and $(z_j,w_j)\in\Omega\times \Omega$ for any $j>0$, so the conditions of the Lemma \ref{aa} are satisfied with $\alpha_j\rightarrow 0$.

We set $\psi_j(z,w)=(1/2\alpha_j)|z-w|^2$ and we apply Jensen-Ishii's maximum principle to the function 
$$(z,w)\mapsto u_{\varepsilon}(z)-v(w)-\psi_j(z,w).$$
Assume its maximum is achieved in $(z_j,w_j)$. Then the following holds.
\begin{lem}For any $\gamma>0$, we can find $(\underline{p},\underline{Q}),(\overline{p},\overline{Q})\in \C^n\times \mathbb{S}^{2n}$ such that
\begin{itemize}
\item $(\underline{p},\underline{Q})\in\bar{J}^{2,+}u_{\varepsilon}(z_j),(\overline{p},\overline{Q})\in\bar{J}^{2,-}v(w_j)$, where 
$\underline{p}=\frac{z_j-w_j}{\alpha_j}=-\overline{p}.$
\item The block diagonal matrix with entries $(\underline{Q}, \overline{Q})$ satisfies 
$$-(\gamma^{-1}+||A||)I\leq \begin{pmatrix}
 \underline{Q}& 0 \\
0 & - \overline{Q}
\end{pmatrix}\leq A+\gamma A^2,$$
\end{itemize}
where $A=D^2\psi_j(z_j,w_j)$ which means 
$$A= \gamma^{-1}\begin{pmatrix}
 I& -I \\
I & I
\end{pmatrix},$$
and $||A||$ is the spectral radius of $A$.
\end{lem}
We apply above lemma for $\gamma=\alpha_j$, we get 
\begin{equation}\label{inequality}
-(2\alpha_j^{-1})I\leq \begin{pmatrix}
 \underline{Q}& 0 \\
0 & - \overline{Q}
\end{pmatrix}\leq \frac{3}{\alpha_j}\begin{pmatrix}
 I& -I \\
I & I
\end{pmatrix},
\end{equation}
From this inequality we deduce that the eigenvalues of $\underline{Q}$, $\overline{Q}$ are $ |\lambda_k|\leq \frac{C}{\alpha_j}$, $k=1,...,2n$. Evaluating the inequality (\ref{inequality}) on vector of the forms $(Z,Z)$ we deduce that 
\begin{equation}\label{quadratic}
 \underline{Q} \leq \overline{Q}
\end{equation}
in the sense of quadratic forms.

For a fixed $Q\in\mathbb{S}^{2n}$, denote by $H= Q^{1,1}$ its (1,1)-part. It is a Hermitian matrix. Since $(\underline{p},\underline{Q})\in\bar{J}^{2,+}u_{\varepsilon}(z_j)$, we deduce from the viscosity differential
inequality satisfied by $u$ that the Hermitian part  $\underline{H}$ of $\underline{Q}$  is positive definite. Applying to the inequality (\ref{quadratic}) vectors $(Z,Z)$ and $(iZ,iZ)$ and adding them, we conclude that the Hermitian part of $\underline{Q}$ and $\overline{Q}$ satisfy the following inequality

 \begin{equation}\label{Hermitian}
 \underline{H}\leq \overline{H}.
 \end{equation}
 The last inequality and the positivity  of $\underline{H}$ forces
 $\overline{H}\geq 0$, thus 
\begin{equation}\label{res1}
0 \leq \Lambda_1(\underline{H})\leq\Lambda_1(\overline{H}).
\end{equation}
From the inequality (\ref{strict}) we have 
\begin{equation}\label{res2}
\Lambda_1(\underline{H})\geq f(z_j)+\frac{\varepsilon}{2}.
\end{equation}
As $v$ is a supersolution, we have
\begin{equation}\label{res3}
\Lambda_1(\overline{H})\leq f(w_j).
\end{equation}
Combining (\ref{res1}), (\ref{res2}) and (\ref{res3}) we have
 $$f(z_j)+\frac{\varepsilon}{2}\leq f(w_j).$$
Since $(z_j,w_j)\rightarrow(z_0,z_0)$ by the continuity of $f$, we get 
$$f(z_0)+\frac{\varepsilon}{2}\leq f(z_0)$$
 and this is a contradiction. Thus, (\ref{maximum}) is not true.
 
 In the general case, for all $\varepsilon>0$, we denote by $\tilde{u}_{\varepsilon}=u(z)+\varepsilon(|z-z_0|^2-B)$, where $B$  is chosen such that $\tilde{u}_{\varepsilon}\leq u$ on $\partial\Omega$. By the above argument, we get $\tilde{u}_{\varepsilon}\leq v$ in $\Omega$ for all $\varepsilon>0$. By letting $\varepsilon\searrow 0$, we get $u\leq v $ in $\Omega$, which completes the proof.
\subsection{Proof of the main theorem}
\begin{thm}\label{en}Let $\varphi$ be a continuous function on $\partial\Omega$. Let  $\underline{\phi}$ be a  subsolution and  $\overline{\phi}$ supersolution of (\ref{first}) such that
$$ \underline{\phi}_{*}=\overline{\phi}^{*}=\varphi\; \text{on}\; \partial\Omega.$$
Then the envelope of subsolutions given  as follows:
$$u(z):=\sup \{v(z)\in L^{\infty}(\Omega)\cap USC(\Omega): v\; \text{is a subsolution of (\ref{first})},\; \underline{\phi}\leq v\leq \overline{\phi} \},$$
  is the unique viscosity solution to the problem (\ref{first}) such that $u=\varphi $ on $\partial \Omega.$
\end{thm}
\begin{proof}
We start by proving that the regularization of the upper envelope is a subsolution. For each $j$, we set
$$u_j=(\sup_{k\geq j} v_k)^{*},\;h_{j,l}=\max(v_j,...v_{j+l}).$$
Since the notion of viscosity subsolution is stable under taking maximum,
we deduce that $h_{l}$ is a viscosity subsolution of (\ref{first}). Observe that $u_j=(\sup_{l\geq 0}h_{j,l})^{*}$ and the sequence $h_{j,l}$ is increasing to $u_j$. It follows from Lemma \ref{convergence} that $u_j$ is a viscosity subsolution of (\ref{first}). Note that $u_j$ converge decreasingly to $u^{*}$,  using again Lemma \ref{convergence} to conclude the proof.

Assume by contradiction that  $u_{*}$ is not a supersolution of (\ref{first}), then there exits a lower test function $q$ for $u_{*}$ at some $z_0$ such that
\begin{equation}\label{perron}
 \Lambda_1(D_{\C}^2q )(z_0)>f(z_0).
 \end{equation}
 By the continuity of $(z,Q)\mapsto F(z,Q)$ in its domain and (\ref{perron}), we can find $r>0$ small enough such that 
  $q\leq u_{*}\; \text{ in a ball}\; \B(z_0,r)$ and $$\Lambda_1(q)(z)> f(z) \; \forall\; z\in \B(z_0,r).$$
 
 We set $q_{\varepsilon}(z)=q(z)+c-\varepsilon|z-z_0|^2$, where $c=\frac{\varepsilon r^2}{2}$, we choose $r,\varepsilon>0$ small enough such that
 $$\Lambda_1(q_{\varepsilon})(z)\geq f(z) \; \forall\; z\in \B(z_0,r).$$
 On $|z-z_0|=r$, we have 
 $$q_{\varepsilon}(z)=q(z)+c-\varepsilon r^2\leq u_{*}(z),$$
 therefore the function defined as follows
 $$h(z)=
 \begin{cases}
      \max(u(z),q_{\varepsilon}(z)),& \;\text{in}\;\B(z_0,r), \\
      u(z),& \;\text{in}\;\Omega\setminus \B(z_0,r)
    \end{cases}$$
    is a viscosity subsolution of (\ref{first}) as a consequence of Proposition \ref{glue}. We infer that 
    $h(z)\leq u(z) $ in $\Omega$, hence $q_{\varepsilon}\leq u(z)$. 
    On other hand we have $q_{\varepsilon}(z_0)-u_*(z_0)=c$.
     By choosing a sequence $z_n \rightarrow  z_0$ so that $u(z_n) \rightarrow u_*(z_0)$,
we have $q_{\varepsilon}(z_n) \rightarrow u_{*}(z_0)+ c$. Therefore, for $n$ sufficiently large, we have $q_{\varepsilon}(z_n)> u(z_n)$
and this contradicts the definition of $u$.  Thus we have $u_{*}$ is also a supersolution. 

Now we know that $u^{*}$ is a subsolution and $u_{*}$ is a supersolution. Since $\varphi = \underline{\phi}_*\leq u_{*}\leq u^{*}\leq \overline{\phi}^*= \varphi$ on $\Omega$, by the viscosity comparison principle we get that $u = u_{*}= u^{*}$ is a continuous viscosity solution of (\ref{first}) with boundary value $\varphi$.
\end{proof}
We now study the Dirichlet problem (\ref{main}) for the complex first eigenvalue equation.

\begin{defi} Let $\varphi$ a Dirichlet data for the first eigenvalue equation  (\ref{first}).
\begin{enumerate}
\item We say an upper semi-continuous $u:\overline{\Omega} \mapsto \R$ is a subsolution to Dirichlet problem (\ref{main}) if it is a subsolution of (\ref{first}) in $\Omega$ and satisfies $u\leq \varphi$ on $\partial \Omega$.
\item We say a lower  semi-continuous $v:\overline{\Omega} \mapsto \R$ is a supersolution to Dirichlet problem (\ref{main}) if it is a supersolution  of (\ref{first}) in $\Omega$ and satisfies $v\geq\varphi$ on $\partial \Omega$.
\end{enumerate}
\end{defi}
We are now in the position to proof the second main theorem:\\
Proof of \textbf{Theorem B}. Since $\Omega$ is  $B$-regular, by the last two statements in Theorem \ref{sib} we have an exhaustive function such that
$\psi \in\psh(\Omega)\cap C^{\infty}(\Omega)$, $\psi<0$, $\lim_{z\rightarrow \partial\Omega}\psi(z)=0$ and $\Lambda_1(\psi)\geq 1$ and also the boundary data is  extendable to plurisubharmonic function $\tilde{\varphi}$ on $\Omega$, continuous on $\overline{\Omega}$. We define $$\underline{\phi}(z)= B \psi(z)+\tilde{\varphi}(z),$$ 
 where $B$ is a constant. We have that $\underline{\phi}$ subsolution. Indeed, let $q$ be a  upper test function of $\underline{\phi}$ at $z_0$, this implies that $q-B\psi$ is a upper test function of $\tilde{\varphi}$ at $z_0$. Since $\tilde{\varphi}$ is a plurisubharmonic function, we have $dd^c(q-B\psi)(z_0)\geq 0$ (see \cite[page 1063]{egz11}). Then 
 $$\Lambda_1(D_{\C}^2(q-B\psi))(z_0)\geq 0.$$
 We use now the super-additivity of $\Lambda_1$ with $\Lambda_1(\psi)\geq 1$ to get the result as follows
  \begin{eqnarray*}
\Lambda(D_{\C}^2 q)(z_0)&\geq &\Lambda_1(D_{\C}^2(q-B\psi))(z_0) +B\Lambda_{1}(D_{\C}^2\psi)(z_0)\\
&\geq& f(z_0),
\end{eqnarray*}
if $B$ is large enough.

 For the supersolution we solve the Dirichlet problem
$$\begin{cases}
  \Delta h= 0,& \;\text{in } \Omega,\\
      
     h=\varphi, & \;\text{on}\;\partial\Omega.
    \end{cases}$$
   Then $\overline{\phi}=h $ is a supersolution to (\ref{main}), since 
   $n\Lambda_1(D^2_{\C}h)\leq \Delta h$  on $\Omega$.
   
   Finally, the existence of a solution follows from Perron’s method. We set
   $$ u(z)= \sup\{ w(z) \;\text{is a subsolution of (\ref{main}) on}\;\Omega, \underline{\phi}\leq w\leq \overline{\phi}\}.$$
   As in the argument from Theorem \ref{en}, we have $\underline{\phi}_*$ and $\overline{\phi}^*$ are subsolution and supersolution, respectively, of (\ref{main}). It follows from  Theorem \ref{maximum}
   $$\underline{\phi}^*-\overline{\phi}_*\leq \max_{\partial\Omega}(\underline{\phi}^*-\overline{\phi}_*).$$
   Since $\underline{\phi}$ and $\overline{\phi}$ are continuous and $\underline{\phi} =\overline{\phi}= \varphi$ on $\partial\Omega$, we infer that  $u^{*}\leq u_{*} $ on $\Omega$ and
$u^{*} = u_{*} $ on $\partial \Omega$. Therefore $u = u
^* = u_*$ is a viscosity solution of (\ref{main}). The uniqueness
follows from the comparison principle.
\section{Operators which are comparable to the first eigenvalue operator}
In this section we will study the partial differential equations of the form:
\begin{equation}\label{general}
G(D_{\C}^2u)=f(z) \; \text{in}\;\Omega,
\end{equation} 
where $\Omega$ is a bounded domain in $\C^n$, $f: \overline{\Omega} \mapsto\R$ is a positive continuous function and $G:\mathcal{C}\mapsto \R^{+}$ satisfies assumptions which will be stated in the sequel.

 For $A \in \mathbb{H}^n$, its eigenvalues will always be ordered as follows:
$$\Lambda_1(A) \leq \Lambda_2(A) \leq \cdots \leq \Lambda_n(A).$$
We then introduce $\Lambda:\mathbb{H}^n\longrightarrow \R^n$ defined by
$$\Lambda(A)=(\Lambda_1(A),\ldots,\Lambda_n(A)).$$
Throughout this section we will always make the following assumptions:
\begin{description}
\item[a)] Domain: $\mathcal{C}\subset\mathbb{H}^n$ is a nonempty open cone  such that $\mathcal{C}_n\subset \mathcal{C}$,
where $$\mathcal{C}_n=\{A\in \mathbb{H}\;|\; A>0\}.$$
  \item[b)] Continuity: $G\in C(\bar{\mathcal{C}}).$
  \item[c)]Comparison: For all $A\in \mathcal{C}$ there exists $C>0$ such that 
  $$G(A+P)-G(A)\geq C \Lambda_1(P),\;\; \forall P\in \mathcal{C}_n.$$
\end{description}
\begin{rem} The assumption \textbf{c)} implies that $G$ is necessarily elliptic in $\mathcal{C}$.
\end{rem}
\begin{rem}We observe that every function which is concave, 1-homogeneous and comparable to the first eigenvalue i.e.  $$G(P)\geq C\Lambda_1(P),\;\forall P\in C_n$$ satisfies the assumption \textbf{c)}. Indeed, the concavity implies $$
  G\left(\frac{A}{2}+\frac{P}{2}\right)\geq \frac{1}{2} G(A)+\frac{1}{2}G(P),$$
 we use now the  the 1-homogeneity and comparability to the first eigenvalues operator to  conclude \textbf{c)}. This observation implies that many complex operators satisfy the assumption \textbf{c)} as we will see in the sequel.
\end{rem}

A function $G:\mathcal{C}\longrightarrow \R$ is said to be a Hessian operator if there exist a set $\Gamma \subset \R^n$ and a function $\hat{G}:\Gamma \longrightarrow \R$ such that
$$\mathcal{C}=\Lambda^{-1}(\Gamma), \quad G(A)=\hat{G}(\Lambda(A)), \quad \forall A \in \mathcal{C}.$$
The notation $\Gamma$ will always be saved for sets of $\R^n$ whereas the notation $\mathcal{C}$ will always be saved for sets of $\mathbb{H}^n$.

Let us now present some examples covered by our framework.
\begin{itemize}
\item
\textbf{The complex Monge-Amp\`ere:}
$$
\hat{G}(\Lambda)=\left(\prod_{i=1}^n \Lambda_i\right)^{\frac{1}{n}},
\quad
\Gamma=\Gamma_n,
$$ where $\Gamma_n=\{\Lambda\in \R^n\quad \Lambda_i>0\quad \forall i\in \{1,..,n\}\}$. 
\item
\textbf{The complex $k$-Hessian:} for $k \in \{1,\ldots,n\}$,
$$
\hat{G}=(\sigma_k)^{\frac{1}{k}},
\quad
\Gamma=\Gamma_k,
$$
where
$$
\Gamma_k=\{\Lambda \in \R^n: \;\; \sigma_{l}(\Lambda_1,..,\Lambda_n)>0, \quad \forall l \in \{1,..,k\}\}
$$
 and $\sigma_k$ is the $k$-th elementary symmetric polynomial:
$$
\sigma_k(\Lambda)=
\sum_{ 1\leq i_1<..<i_k\leq n} \Lambda_{i_1} \cdots \Lambda_{i_k}.
$$

\item
For $s \in [0,1]$,
$$
\hat{G}(\Lambda) =((1-s)^2\Lambda_1 \Lambda_2+s(\Lambda_1+\Lambda_2)^2)^{\frac{1}{2}}
, \quad \Gamma=\Gamma_{2-s},$$
where 
$
\Gamma_{2-s}=\{\Lambda \in \R^2 \;\; \Lambda_1+s\Lambda_2>0, \quad s\Lambda_1+\Lambda_2>0\}.
$
\item
\textbf{The complex $k$-Monge-Amp\`ere :} for $k \in \{1,\ldots,n\}$,
$$
\hat{G}=\left(
\prod_{i_1<\ldots<i_k} \left(\Lambda_{i_1}+\ldots+\Lambda_{i_k}\right)\right)^{\frac{1}{C^{n}_{k}}},
\quad
\Gamma=\Gamma_k',
$$
$
\text{where}\quad
\Gamma_k'=
\{\Lambda \in \R^n:\quad \Lambda_{i_1}+\ldots+\Lambda_{i_k}>0, \quad \forall  \;i_1<\ldots<i_k\}\\
$
and $C_n^k=\frac{n!}{k!(n-k)!}.$
\item
\textbf{The complex eigenvalues:} for $k \in \{1,\ldots,n\}$,
$$
\hat{G}=\Lambda_k,
\quad
\mathcal{C}=\{\Lambda \in \R^n: \quad \Lambda_k>0\}.$$
\item
\textbf{The complex linear combination of eigenvalues:} 
$$
\hat{G}=\sum_{k=1}^{n}a_k\Lambda_k,
\quad
\Gamma=\Gamma_a,$$
where $\Gamma_a=\{\Lambda \in \R^n: \quad \sum_{k=1}^{n}a_k\Lambda_k>0\}$, $a=(a_1,...,a_n)\in\R^n$ such that $a_k\geq 0$\; $\forall k \in\{1,...,n\}$ and $\sum_{k=1}^n a_k>0.$
\end{itemize}
\begin{rem} The last two examples satisfy the assumption \textbf{c)} from the second statement of the Lemma \ref{eigenvalues}. For more details about these operators in the real setting see \cite{hl13,ferrari}.
\end{rem}
\begin{rem}The first four examples are concave, 1-homogeneous and comparable to Monge-Amp\`ere operator i.e.
\begin{equation}\label{op}
G(P)\geq C \det(P)^{\frac{1}{n}}, \forall P\in \mathcal{C}_n
\end{equation}
(see \cite{ADO20,AO21}).
 Since the Monge-Amp\`ere operator is comparable to the first eigenvalue operator i.e.
 \begin{equation}
  \det(P)^{\frac{1}{n}}\geq C\Lambda_1(P), \forall P\in \mathcal{C}_n.
   \end{equation}
 Then they are comparable to the first eigenvalue operator.
\end{rem}
\begin{rem}The operators $\Lambda_k$ for $k\in \{2,..,n\}$ are not concave, which means that our framework contains more than the concave operators.
\end{rem}
After identifying $\C^n$ with $\R^{2n}$ in the standard way we define $H: \Omega\times \mathbb{S}^{2n} \rightarrow [-\infty,+\infty)$ by 
$$H(z,Q):=
 \begin{cases}
  G(Q^{1,1})-f(z)  ,& \;\text{if}\;\; Q^{1,1}\in\overline{\mathcal{C}},\\
      -\infty, &\;\text{otherwise}\;\;.
    \end{cases}$$
   Where
$Q^{1,1}\in \mathbb{H}^n$ is the Hermitian (1,1)-part of $Q\in \mathbb{S}^{2n}$ considered as a real quadratic form on $\C^n$.

We define the subsolution and supersolution of (\ref{general}) as follows.
\begin{defi} An upper semi-continuous function  $u: \Omega \rightarrow \R\cup \{-\infty\}$ is a viscosity sub-solution to (\ref{general})
 if $u\neq -\infty$ and for any $z_0\in\Omega$ and any $C^2$ function $\varphi$ touches $u$ from above at $z_0$ we have
 $$G(D_{\C}^2\varphi)\geq f(z)\;\text{at}\; z_0.$$
 
 A lower semi-continuous function  $u: \Omega \rightarrow \R\cup \{-\infty\}$ is a viscosity super-solution to (\ref{general})
 if $u\neq -\infty$ and for  any $z_0$ and any $C^2$ function $\psi$ touches $u$ from below at $z_0$ and $D_{\C}^2\psi(z_0)\in\bar{\mathcal{C}}$ we have
 $$G(D_{\C}^2\psi)\leq f(z)\; \text{at}\; z_0.$$
 
 A function $u: \Omega \rightarrow \R$ is a viscosity solution to (\ref{general}) if it is both a viscosity subsolution and a viscosity
supersolution to (\ref{general}). A viscosity solution is automatically a continuous function in $\Omega$.
\end{defi}
From the fact that $G$ is a continuous function on its domain and from the definition of differential jet and its approximate differential jet, we can reformulate the definition of a subsolution and supersolution as follows.
\begin{prop}\label{2}
An upper semi-continuous function $u : \Omega \mapsto \R$ is a viscosity subsolution of (\ref{general}) if and only if 
for all $z_0\in\Omega,\; (p,Q)\in\bar{J}^{2,+}(z_0),$
$$G(Q)\geq f(z_0).$$
																			
A lower semi-continuous function $u : \Omega \mapsto \R$ is a viscosity supersolution of (\ref{general}) if and only if for all $z_0\in\Omega$, $(p,Q)\in\bar{J}^{2,-}(z_0)$ and $Q\in\bar{\mathcal{C}}$
$$G(Q)\leq f(z_0).$$
\end{prop}
\begin{thm}Let $\mathcal{C}$ and $G$ satisfy the above assumptions \textbf{a)}, \textbf{b)} and \textbf{c)}. Let $u$ be a bounded viscosity subsolution and $v$ be a bounded viscosity supersolution of (\ref{general}). If $u\leq v$ on $\partial\Omega$ then $u\leq v$ on $\Omega$. 
\end{thm}
\begin{proof}The proof is almost the same as the proof of the comparison principle for the first eigenvalue. We will explain only where we use our assumptions. We use now our assumption \textbf{c)} we have
\begin{eqnarray}\label{strict2}
G(D_{\C}^2u_{\varepsilon})(z)&=&G \left(D_{\C}^2\left(u(z)-\varepsilon|z-z_0|^2+\frac{\varepsilon}{2}|z-z_0|^2\right)\right)\nonumber\\
&\geq &  G\left(D_{\C}^2\left(u(z)-\varepsilon|z-z_0|^2\right)\right)+\Lambda_1\left(\frac{\varepsilon I}{2}\right)\nonumber\\
&\geq & f(z)+\frac{\varepsilon}{2}.
\end{eqnarray}
From the inequality (\ref{Hermitian}) and ellipticity  of $G$, we have 
\begin{equation}
 G(\underline{H})\leq G(\overline{H}).
\end{equation}
From the inequality (\ref{strict2}) we have 
\begin{equation}
G(\underline{H}) \geq f(z_j)+\frac{\varepsilon}{2}.
\end{equation}
As $v$ is a supersolution we have
\begin{equation}
G(\overline{H})\leq f(w_j).
\end{equation}
 The proof now continues as in the proof of the comparison principle for the first eigenvalue.
\end{proof}
\begin{thm}
Let $\mathcal{C}$ and $ G$ satisfy the above assumptions \textbf{a)}, \textbf{b)} and \textbf{c)}. Let $\varphi$ be a continuous function on $\partial\Omega$. Let  $\underline{\phi}$ be a  subsolution and  $\overline{\phi}$ supersolution of (\ref{general}) such that
$$ \underline{\phi}_{*}=\overline{\phi}^{*}=\varphi\; \text{on}\; \partial\Omega.$$
Then the envelope of subsolutions is given  as follows:
$$u(z):=\sup \{v(z)\in L^{\infty}(\Omega)\cap USC(\Omega): v\; \text{is a subsolutin of (\ref{first})},\; \underline{\phi}\leq v\leq \overline{\phi} \}.$$
 $u$ is the unique viscosity solution to the problem (\ref{general}) such that $u=\varphi $\; on $\partial \Omega.$
\end{thm}
\begin{proof}The proof is the same as the proof of Theorem \ref{en}. The things which make the proof work are the continuity of $H$ and the comparison principle.
\end{proof}
\section*{Acknowledgements}
\small{The author was supported by the Polish National Science Centre Grant 2017/26/E/ST1/00955. The author would like to thank S{\l}awomir Dinew for numerous discussions, suggestions and for reading the first version of this paper.}
\bibliographystyle{amsalpha}
\bibliography{abja}
\end{document}